\documentclass[microtype]{gtpart}
\usepackage[all]{xy}
\usepackage[parfill]{parskip}
\usepackage{amsmath}
\usepackage{amscd, latexsym, hyperref, rlepsf, times}
\usepackage{color,graphicx}
\newtheorem{dummy}{anything}[section]
\newtheorem{theorem}[dummy]{Theorem}

\newtheorem{proposition}[dummy]{Proposition}

\theoremstyle{definition}

\newtheorem{remark}[dummy]{Remark}

\renewcommand
{\theequation}{\thesection.\arabic{dummy}}

\DeclareMathOperator{\id}{id} 

\newcommand{\f}[1]{\mathbb{#1}}

\newcommand{\QED}{\vspace{-.32in}\begin{flushright}\qed\end{flushright}}
 \newcommand{\sgn}{\operatorname{sgn}}
 
\hypersetup{
pdfauthor = {Yank\i\ Lekili},
pdftitle = {Planar open books with four binding components},
pdfsubject = {Geometric Topology, Symplectic Geometry},
pdfkeywords = {open books, contact structures, braids, Dehn surgery}}

\begin{document}
\title[]
{\ \ \ \  Planar open books with four binding components}
\author[Lekili]{Yank\i\ Lekili}
\address{Max-Planck Institut f\"ur Mathematik, Bonn, Germany}
\email{ylekili@mpim-bonn.mpg.de} 

\begin{abstract} 

We study an explicit construction of planar open books with four binding components on
any three-manifold which is given by integral surgery on three component pure braid
closures. This construction is general, indeed any planar open book with four binding
components is given this way. Using this construction and results on exceptional
surgeries on hyperbolic links, we show that any contact structure of $S^3$ supports a
planar open book with four binding components, determining the minimal number
of binding components needed for planar open books supporting these contact
structures. In addition, we study a class of monodromies of a planar open book with
four binding components in detail. We characterize all the symplectically fillable
contact structures in this class and we determine when the Ozsv\'ath-Szab\'o contact
invariant vanishes.  As an application, we give an example of a right-veering
diffeomorphism on the four-holed sphere which is not destabilizable and yet supports
an overtwisted contact structure. This provides a counterexample to a conjecture of
Honda, Kazez, Mati\'c from \cite{HKM3}.  \end{abstract}

\maketitle

\let\thefootnote\relax\footnotetext{I would like to thank Andy Wand for helpful
conversations, Tolga Etg\"u and John Etnyre for useful comments on a previous draft. I
also acknowledge Max-Planck Institut f\"ur Mathematik for the support. }

\section{Introduction}

Let $Y$ be a closed oriented 3--manifold and $\xi$ be a contact structure on $Y$.
Recall that an open book is a fibration $\pi: Y - B \to S^1$ where $B$ is an oriented
link in $Y$ such that the fibres of $\pi$ are Seifert surfaces for $B$. The contact
structure $\xi$ is said to be supported by an open book $\pi$ if $\xi$ is the kernel
of a one-form $\alpha$ such that $\alpha$ evaluates positively on the positively
oriented tangent vectors of $B$ and $d\alpha$ restricts to a positive area form on
each fibre of $\pi$. The fibres of $\pi$ are called \emph{pages} of the open book. We
will consider abstract open books $(S,\phi)$ where $S$ is a page of the open book, and
$\phi \in Diff(S,\partial S)$. It is easy to construct an open book as above, starting
from the data $(S,\phi)$ (see \cite{etnyrelectures}).  

It is well known that every contact structure $\xi$ is supported by an open book on
$Y$ and all open book decompositions of $Y$ supporting $\xi$ are equivalent up to
positive stabilizations and destabilizations \cite{G}. In light of this theorem, to
study contact structures, we will study abstract open books $(S,\phi)$ supporting
them. We should note that in our case the right notion of equivalence provided by the
Giroux's theorem is contact isomorphism (not contact isotopy, see \cite{etnyrelectures}).

In \cite{etnyre}, Etnyre proved that every overtwisted contact structure is supported
by a planar open book. On the other hand, there are known obstructions for a tight
contact structure to admit a supporting planar open book, \cite{etnyre}, \cite{OSS},
\cite{wand}. 

For a contact structure $(Y,\xi)$, in $\cite{EtOz}$, Etnyre and Ozbagci defined
invariants of $\xi$ by a measure of topological complexity of its supporting open
books. We recall these here: \[ \text{sn}(\xi) = \text{min}\{ -\chi(\pi^{-1}(\theta))|
\pi : Y-B \to S^1 \text{\ supports\ } \xi \} \] \[ \text{sg}(\xi) = \text{min}\{
g(\pi^{-1}(\theta))| \pi : Y-B \to S^1 \text{\ supports\ } \xi \} \] \[ \text{bn}(\xi)
= \text{min}\{ |B| | \pi : Y-B \to S^1 \text{\ supports\ } \xi \text{\ and \ }
g(\pi^{-1}(\theta)) = sg(\xi) \} \ , \] where $\theta$ is any point in $S^1$ , $g(.)$
is the genus, and $|.|$ is the number of components.

These are called \emph{support norm, support genus} and \emph{binding number} in the
order given above. In general, it is hard to compute these invariants for a given
$\xi$. From the above definition, it is easy to see that $\text{sn}(\xi) \leq
2 \text{sg}(\xi) + \text{bn}(\xi)-2 $, however it is known that in general these
invariants are independent of each other (\cite{EtL}, \cite{balEt}).

In this article, we will determine all of these invariants for all the contact
structures on $S^3$. Previously for any contact structure $\xi$ on $S^3$, Etnyre and
Ozbagci showed that $sg(\xi)=0$, $bn(\xi) \leq 6 $ and $sn(\xi) \leq 4$. Recall that,
there exists a unique tight contact structure on $S^3$ having $d_3 = -\frac{1}{2}$. It
is easy to show that this is supported by the open book $(D^2, \id)$, hence $sg=0$,
$bn=1$ and $sn=-1$ for the tight contact structure on $S^3$. The overtwisted contact
structures on $S^3$ are classified by their $d_3$ invariants which takes values in
$\f{Z} + \frac{1}{2}$. We will write $\xi_n$ for the overtwisted contact structure on
$S^3$ with $d_3 =n$. Our first result determines the invariant of these:

\begin{theorem} 
\label{theorem1}
Let $\xi_n$ be the overtwisted contact structure on $S^3$ with $d_3(\xi_n)=n$, then $sg(\xi_n)=0$ for
all $n$,\[ bn(\xi_\frac{1}{2}) = 2 \] \[ bn(\xi_{-\frac{1}{2}})= bn(\xi_{\frac{3}{2}})= 3\] 
\[bn(\xi_n)=4 \text{\ for all \ } n \neq - \frac{1}{2}, \frac{1}{2}, \frac{3}{2} \]
\[sn(\xi_\frac{1}{2}) = 0 \]
\[sn(\xi_{-\frac{1}{2}}) = sn(\xi_{\frac{3}{2}}) = 1  \]
\[sn(\xi_n) = 2 \text{\ for all \ } n \neq - \frac{1}{2}, \frac{1}{2}, \frac{3}{2} \]
 \end{theorem} 

Note that the results for $n = - \frac{1}{2}, \frac{1}{2}, \frac{3}{2}$ were
calculated by $\cite{EtOz}$ via an easy classification of planar open books with three
or less boundary components, which we review here. Let $(Y,\xi)$ be the contact
three-manifold supported by $(S,\phi)$. Below, we write $(Y,\xi_{st})$ to denote the
unique tight contact structure on $Y$ whenever $Y$ has a unique tight contact
structure. These descriptions and more can be found in \cite{EtOz}.

\begin{itemize} 

\item If $S = D^2$ , then $\phi = \id$ and $(Y,\xi) = (S^3, \xi_{st})$.  \item For $S
= S^1 \times [0,1]$, let $a$ denote the simple closed curve generating $H_1(S)$. If
$\phi = \tau_a^p$ , then $(Y,\xi) = (L(p,p-1),\xi_{st}) $ for $p>0$, $(Y,\xi)=(S^1
\times S^2 , \xi_{st})$ for $p=0$, and $(Y,\xi)= (L(p,1), \xi)$ for $p<0$, where $\xi$
is overtwisted with $e(\xi)=0$ and $d_3(\xi) = \frac{3-p}{4} $.  Note that $S^3$
appears exactly for $p=\pm 1$. For $p=1$, this is a stabilization of the standard open
book of tight contact structure in $S^3$, and for $p=-1$, we get the overtwisted
contact structure $\xi_{\frac{1}{2}}$.

\item When $S$ has three boundary components, let $a$, $b$, $c$ denote boundary
parallel simple closed curves.  If $\phi = \tau_a^p \tau_b^q \tau_c^r$, then $Y$ is
the Seifert fibered space with $e_0 = \lfloor -\frac{1}{p} \rfloor + \lfloor
-\frac{1}{q} \rfloor + \lfloor -\frac{1}{r} \rfloor$ as shown in Figure \ref{Figure1}.
We only note that it is easy to draw a contact surgery diagram of these contact
structures \cite{EtOz}. The authors calculate exactly when $S^3$ has such an open
book, it turns out all of these open books support either $\xi_{-\frac{1}{2}}$ or
$\xi_{\frac{3}{2}}$. 

\end{itemize}

To determine $bn(\xi_n)$ for the remaining cases, we simply construct planar open
books with four binding components supporting $\xi_n$ for the remaining cases. This
determines $bn(\xi_n)$. To calculate $sn(\xi_n)$, we show that none of these contact
structures can be supported by an open book with page a torus with one boundary
component. 

\begin{figure}[!h]
\centering
\includegraphics[scale=0.5]{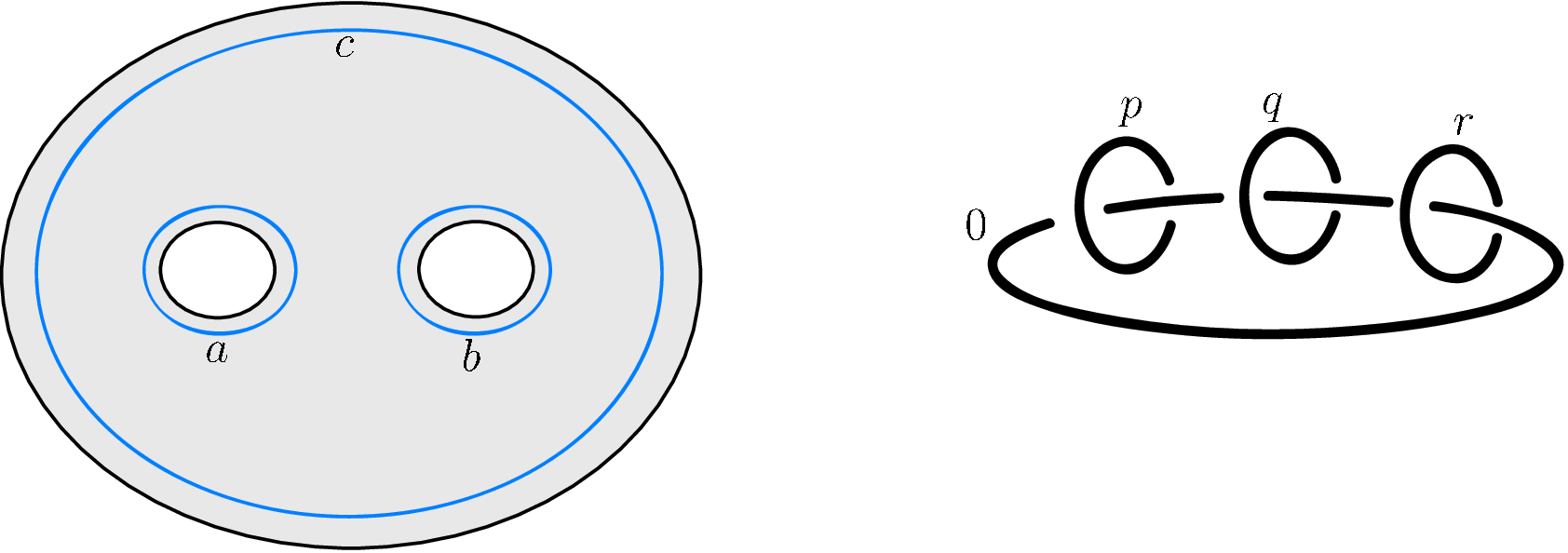} 
\caption{Open books with page a three-holed sphere}
\label{Figure1}
\end{figure}

In \cite{HKM1}, Honda, Kazez and Mati\'c proves that a contact structure $\xi$ is
tight if and only if all of the open book decompositions $(S,\phi)$ supporting $\xi$
have right-veering monodromy $\phi \in Aut(S,\partial S)$. This result is useful in
proving that $\xi$ is overtwisted by exhibiting a supporting open book with a
monodromy which is not a right-veering diffeomorphism. On the other hand, when $S$ is a punctured torus, the
same authors in \cite{HKM3} also prove that the supported contact structure is tight
if and only if the \emph{given} monodromy is right-veering. In general, however a
right-veering diffeomorphism does not always correspond to a tight contact structure.
In fact, any open book can be stabilized to a right-veering one. However, Honda, Kazez
and Mati\'c optimistically conjecture that if the monodromy is given by a right
veering diffeomorphism that does not admit a destabilization (in the sense of Giroux
stabilization) then the supported contact structure is tight. Our next result gives a
counterexample to this conjecture:
\begin{theorem}
\label{theorem2}
There exists an open book $(S,\phi)$ on the Poincar\'e homology sphere
$\Sigma(2,3,5)$ where $S$ is a four-holed sphere and $\phi = \tau_a^5 \tau_b^2 \tau_c
\tau_d \tau_e^{-2}$ which
is right-veering and not destabilizable such that the supported contact structure is
an overtwisted contact structure. 
\end{theorem}

\begin{figure}[!h]
\centering
\includegraphics[scale=0.7]{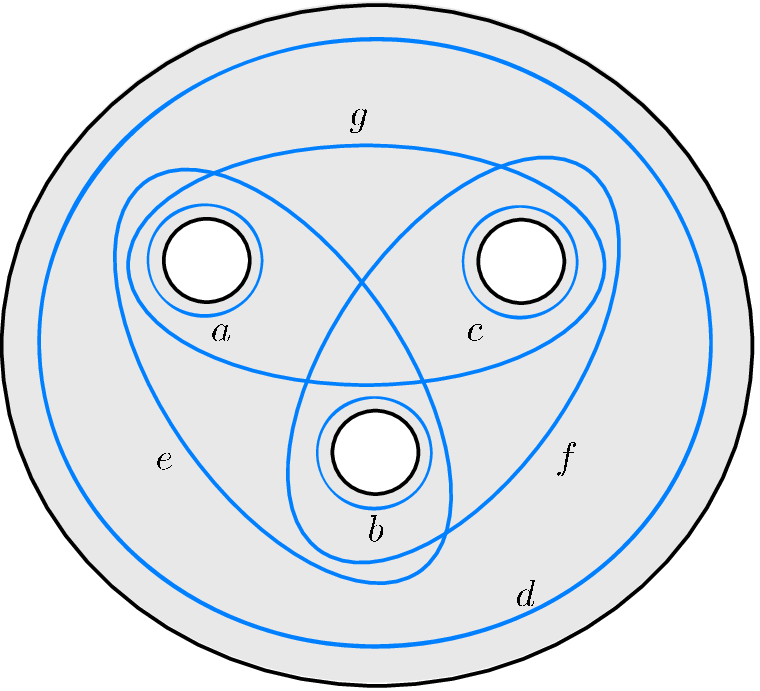} 
\caption{Generators of the mapping class group of four-holed sphere}
\label{Figure2}
\end{figure}

Of independent interest, we also prove the following characterization concerning positive
factorizations of a family of elements in the mapping class group of four-holed
sphere. We denote by $c^+(\phi)$ the Ozsv\'ath-Szab\'o contact invariant of the
contact structure supported by $(S,\phi)$, where $S$ is the four-holed sphere.

\begin{theorem}
\label{theorem3}
Let $\phi = \tau_a^\alpha \tau_b^\beta \tau_c^\gamma \tau_d^\delta \tau_e^\epsilon
\tau_f^\eta$, then $\phi$ admits a positive factorization if and only if
$\min\{\alpha,\beta,\gamma,\delta\} \geq \max \{-\epsilon, -\eta, 0\}$. 
Furthermore, this latter condition is satisfied if and only if $c^+(\phi) \neq 0$.
\end{theorem} 

Note that the results of \cite{wendl} and \cite{nieder} together with the above
proposition imply that the contact structure supported by $(S,\tau_a^\alpha
\tau_b^\beta \tau_c^\gamma \tau_d^\delta \tau_e^\epsilon \tau_f^\eta) $ admits a Stein
filling (or equivalently a weak-symplectic filling) if and only if
$\min\{\alpha,\beta,\gamma,\delta\} \geq \max \{-\epsilon, -\eta, 0\}$. An interesting
question left open is whether all non-fillable contact structures in the class of
monodromies considered above are overtwisted.  Note that one can easily show that some
monodromies give overtwisted contact structures by showing that they are not
right-veering, however Theorem \ref{theorem2} shows that right-veering restriction by
itself is not enough to answer this question.

We pause here to declare our conventions for the rest of the paper. We denote by
$\tau_a$ a right handed Dehn twist around the curve $a$. We will adhere to braid
notation for compositions: $\tau_a \tau_b$ means applying a right handed Dehn twist
around $a$ first and then a right handed Dehn twist around $b$. We will also use the
following conventions for braid groups: Our braids will be drawn from top to bottom
with the strands numbered $1,2,\ldots, n$ from left to right. The convention for
positive and negative half twist is as shown below.

\begin{figure}[!h]
\centering
\includegraphics[scale=0.6]{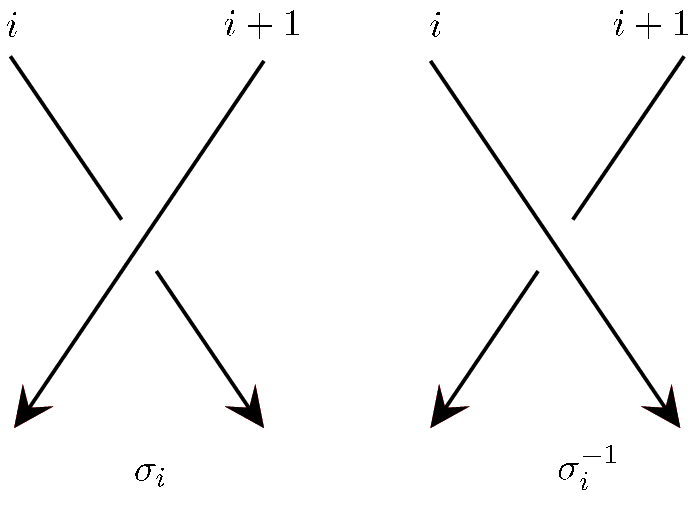} 
\caption{Braid group generators}
\label{Figure3}
\end{figure}

\section{A topological study of planar open books}
\subsection{Planar open books and Dehn surgery}

We first recall a classical proposition relating the mapping class group of a
$n$-holed disk with Dehn surgery on pure braids (see for example \cite{prasolov} for more
than presented here). Let $\phi$ be a diffeomorphism of the
$n$-holed disk, identity on the boundary. This diffeomorphism can be extended to a
diffeomorphism $\tilde{\phi}$ of the disk simply by extending $\phi$ by identity.
Since any diffeomorphism of the disk identical on the boundary is isotopic to
identity, there exists an isotopy $\tilde{\phi}_t$ such that $\tilde{\phi}_0= \id$ and
$\tilde{\phi}_1=\tilde{\phi}$ . Let $x_1,\ldots,x_n$ be points in the disks that fills
the holes, then we obtain a pure braid $\beta(\phi)$ by considering the union of arcs $(t,
\tilde{\phi}_t(x_i))$ in $D^2 \times [0,1]$, $t \in [0,1]$ (see Figure \ref{Figure4}
for an example). This pure braid almost
captures the whole $\phi$, except $\phi$ can have extra boundary twists around the
holes. We summarize this in the proposition below. Let $D_n$ denote the $n$ holed
disk, and $\text{Map}(D_n, \partial D_n)$ be the mapping class group of
diffeomorphisms which are identical on the boundary. Let $P_n$ be the pure braid group
on $n$ strands. 

\begin{proposition} \label{braid}
$ Map (D_n , \partial D_n) = P_n \times \f{Z}^n $ \QED
\end{proposition}

\begin{figure}[!h]
\centering
\includegraphics[scale=1]{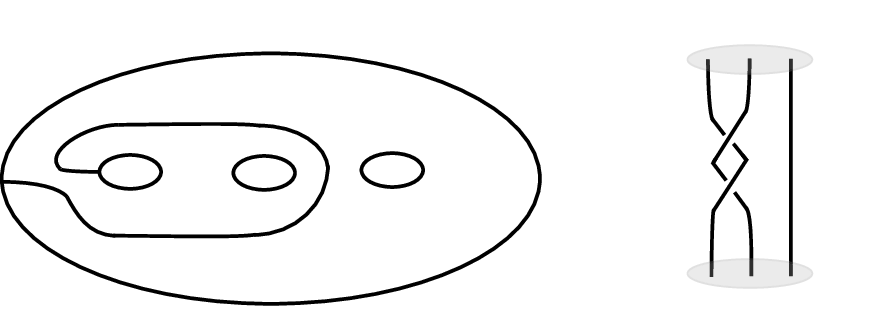} 
\caption{Pure braid associated with a mapping class}
\label{Figure4}
\end{figure}

Note that $n$-holed disk is topologically the same as $n+1$-holed sphere, however the
above isomorphism is meaningful only after choosing a boundary component of the
$n+1$-holed sphere to be identified with the boundary of $D^2$ after filling in the
other boundary components with disks. Nevertheless, such a choice can be made once and
for all. By looking at Figure \ref{Figure2}, we choose the boundary component parallel
to the curve $d$ to correspond to the boundary of $D^2$, and the pure braid will be
obtained by filling the boundary components parallel to the curves $a$, $b$ and $c$,
in addition we choose the ordering of the strands of the pure braid in this order. 

This proposition gives us an alternative way to describe the underlying topological
manifold supported by an open book $(D_n, \phi)$. Namely if $Y$ has an open book
$(D_n, \phi)$, then $Y$ is obtained by Dehn surgery on the braid closure
$\hat{\beta}(\phi)$ of the braid $\beta(\phi)$ with surgery coefficients determined by
the above isomorphism. 

In this article, we study planar open books with four binding components. For the sake of
explicitness, we give a more precise statement of the above discussion for this case. 

Let $S= D_3$ denote a four-holed sphere,  the mapping class group $Map(S,\partial S)$
is not a free abelian group in contrast to the case three-holed sphere , in particular
it has a subgroup isomorphic to $F_2$, the free group on two generators, generated by
Dehn twists around $e$ and $f$ in Figure \ref{Figure2}. In fact, it is a classical
fact that $Map(S,\partial S)$ can be seen as a direct product $\f{Z}^4 \times F_2$
(see \cite{FadellNeuwirth}).  We can see this as follows: Because of Proposition
\ref{braid}, it suffices to see that $P_3$ is $\f{Z} \times F_2$. Recall that $P_3$ is
isomorphic to the fundamental group of the space of triples of distinct points on the
plane (\cite{Fox}). Consider the
forgetful map, from $P_3 \to P_2$, given by forgetting about the middle strand.
$P_2$ is $\f{Z}$ and the kernel of this map is $\pi_1(\f{C} - \{-1,1\}, 0)$,
which is $F_2$. Thus we have a short exact sequence: \[ 0 \to F_2 \to P_3 \to \f{Z} \to 0  \] 
where the kernel is generated by $\sigma_1^2$ and $\sigma_2^2$, and the image is
generated by the central element $(\sigma_2\sigma_1\sigma_2)^2$ which corresponds to a
full right-handed twist of the three-strands. Therefore, any pure 3-braid is expressed
uniquely as $(\sigma_2\sigma_1\sigma_2)^{2\delta} \sigma_1^{2\epsilon_1}
\sigma_2^{2\eta_1} \ldots \sigma_1^{2\epsilon_k} \sigma_2^{2\eta_k}$, where $\delta,
\epsilon_i, \eta_i$ are integers.  

Therefore, under the identification of Proposition \ref{braid} any mapping class $\phi \in Map(S, \partial
S)$ can be represented by: \[ \phi = \tau_a^\alpha \tau_b^\beta \tau_c^\gamma
\tau_d^\delta \tau_e^{\epsilon_1} \tau_f^{\eta_1} \ldots \tau_e^{\epsilon_k}
\tau_f^{\eta_k} \]

and such representation is unique. 

Here $\phi$ is identified with the pure braid $\beta(\phi)
=(\sigma_2\sigma_1\sigma_2)^{2\delta} \sigma_1^{2\epsilon_1} \sigma_2^{2\eta_1} \ldots
\sigma_1^{2\epsilon_k} \sigma_2^{2\eta_k}$, and the integers $(\alpha , \beta,
\gamma)$. For such open books, we have the following proposition as part of
the general discussion above:

\begin{proposition} \label{surgery} Let $S$ be the four-holed sphere and $\phi =
\tau_a^\alpha \tau_b^\beta \tau_c^\gamma \tau_d^\delta \tau_e^{\epsilon_1}
\tau_f^{\eta_1}
\ldots \tau_e^{\epsilon_k} \tau_f^{\eta_k}$ . Let $\epsilon = \sum_{i=1}^k \epsilon_i$
and $\eta = \sum_{i=1}^k \eta_i$. Then the topological manifold $Y$ given by the open
book $(S,\phi)$ can be obtained by Dehn surgery on the braid closure of the pure three braid
$\beta = (\sigma_2\sigma_1\sigma_2)^{2\delta} \sigma_1^{2\epsilon_1} \sigma_2^{2\eta_1} \ldots
\sigma_1^{2\epsilon_k} \sigma_2^{2\eta_k}$, with surgery coefficients $(\alpha+\delta+\epsilon,
\beta+\delta+\epsilon+\eta, \gamma+\delta+\eta)$. \QED \end{proposition}  

\subsection{Planar open books on the three-sphere}

We would like to construct planar open books with four boundary components on $S^3$.
We will look for planar open books with simple monodoromy of the form $\phi =
\tau_a^\alpha \tau_b^\beta \tau_c^\gamma \tau_d^\delta \tau_e^\epsilon \tau_f^\eta$.
In light of Proposition \ref{surgery}, we would like to know when a surgery on a
braid closure of a pure three-braid of the form $\beta=(\sigma_2 \sigma_1
\sigma_2)^{2\delta} \sigma_1^{2\epsilon} \sigma_2^{2 \eta}$ yields $S^3$. Fortunately,
this question is completely resolved by Armas-Sanabria and Eudave-Mu$\tilde{\text{n}}$oz in
\cite{eudave} by depending on deep results on
Dehn surgery on knots. In particular, the authors list several infinite families.
Therefore, we can describe precisely when an open book $(S, \tau_a^\alpha \tau_b^\beta
\tau_c^\gamma \tau_d^\delta \tau_e^\epsilon \tau_f^\eta)$ is an open book on $S^3$.

From the list provided in \cite{eudave} we pick a convenient family. By using Kirby
calculus, we will verify independently that these indeed give $S^3$, and our next task
is to calculate the $d_3$ invariants of the contact structures supported by the
corresponding open books. The difficulty is that we would like to see that any value
in $\f{Z}+\frac{1}{2}$ can be achieved. We will apply several tricks to ensure this.
Therefore, as a consequence of these calculations, we show that every contact
structure on $S^3$ is supported by an open book with a planar page with at most four
binding components. 

\emph{Proof of Theorem \ref{theorem1} : }  We will start with the braid $\beta =
(\sigma_2 \sigma_1 \sigma_2)^2 \sigma_1^{-4} \sigma_2^{-4}$. Figure
\ref{Figure5} is a picture of the closure of this braid, also known as the chain link.
The hyperbolic structure on its complement was first constructed by Thurston in his
notes \cite{thurston}, and this manifold has been called as the ``magic manifold'' by
Gordon and Wu \cite{gordon} \cite{gordonwu} as one gets most of the hyperbolic manifolds and
most of the interesting non-hyperbolic fillings of cusped hyperbolic manifolds (see
\cite{martelliPetronio} for a classification of \emph{all} exceptional surgeries on
this link). It is the smallest known hyperbolic manifold with 3 cusps of smallest
known volume and complexity \cite{adamsSherman}.

It is easy to see by blowing down twice that $(-1, -2, -4)$ surgery on this link yields
$S^3$ (see below for the more general case). 

\begin{figure}[!h]
\centering
\includegraphics[scale=0.6]{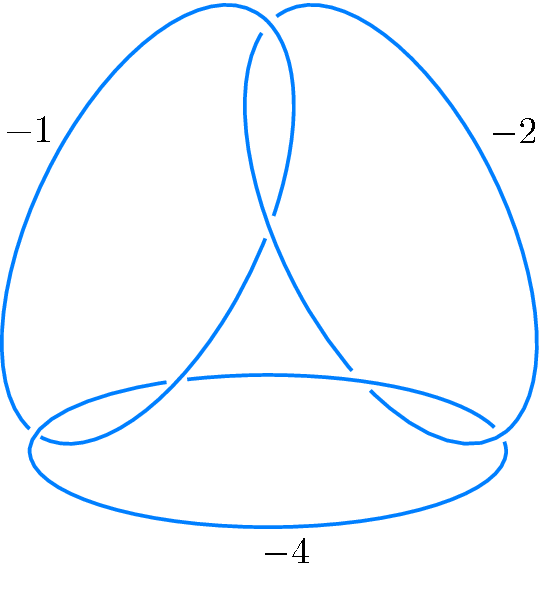} 
\caption{Surgery on the chain link}
\label{Figure5}
\end{figure}

Therefore, by our Proposition \ref{surgery}, it follows that the open book with page
$S$, a four-holed sphere, and $\phi = \tau_a^\alpha \tau_b^\beta \tau_c^\gamma \tau_d
\tau_e^{-2} \tau_f^{-2}$ is an open book on $S^3$ when \[(\alpha-1, \beta-3,
\gamma-1) = (-1,-2,-4) \]

More generally, consider the braid $\beta =(\sigma_2 \sigma_1 \sigma_2)^2
\sigma_1^{-2p} \sigma_2^{-4}$ and perform Dehn surgery with coefficients $(1-p, -p ,
-4)$. In Figure \ref{Figure6}, we verify that we still obtain $S^3$. 

\begin{figure}[!h]
\centering
\includegraphics[scale=0.7]{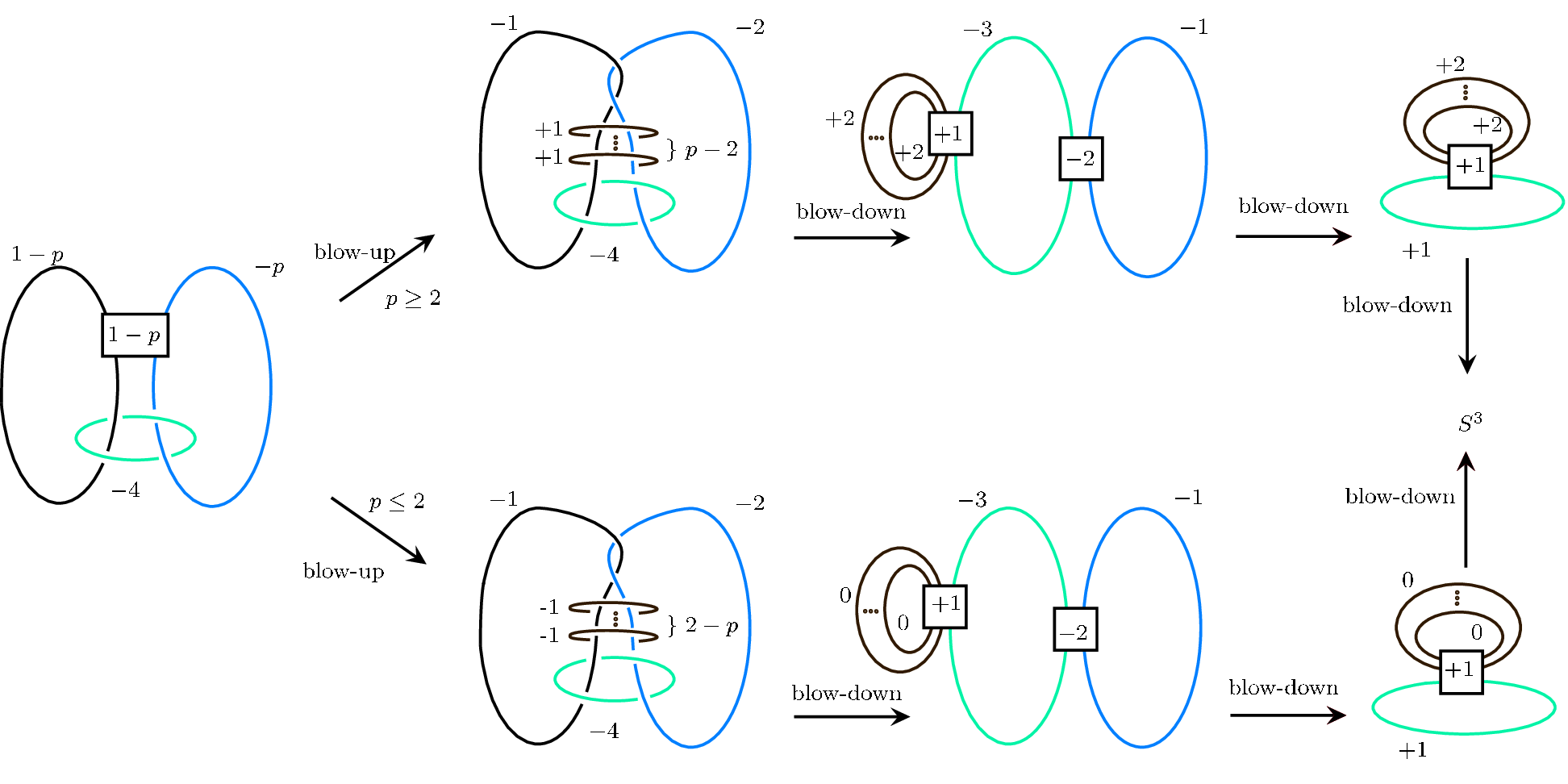} 
\caption{Surgery on a family of links yields $S^3$}
\label{Figure6}
\end{figure}

By reflecting (which amounts to changing orientation), we also know that Dehn surgery
on $\beta =(\sigma_2 \sigma_1 \sigma_2)^{-2} \sigma_1^{2p} \sigma_2^{4}$ with
coefficients $(p-1, p ,4)$ also yields $S^3$.

After some experimentation, the author found that the following two families of open
books (which are obtained from one another by reflecting the braid as above) will be
sufficient for our purposes. (Note that reflecting the braid amounts to changing
the orientation, but since $S^3$ has an orientation reversing diffeomorphism, this
will still give an open book on $S^3$. Though, as we will see below the supported
contact structure will change!)

It will suffice to consider the following two possibilities:
\begin{eqnarray*}
\phi_p &=& \tau_b \tau_c^{-3} \tau_d \tau_e^{-p} \tau_f^{-2} \\
\bar{\phi}_p &=& \tau^{-1}_b \tau_c^{3} \tau^{-1}_d \tau_e^p \tau_f^2\\
\end{eqnarray*}
We will denote the supported contact structures by $\xi_p$ and $\bar{\xi}_p$. Note
that in both open books the monodromies have boundary parallel negative Dehn twists,
it is easy to see that in this case, the monodromies are not right-veering. Therefore, the supported contact structures are
overtwisted.
  
To determine the contact structures, following the description in \cite{EtOz} (see
also \cite{balEt}), we will next compute the $d_3$ invariants of the supported
contact structures from the monodoromy data of the open books. First, we briefly
review the strategy, for more details see $\cite{EtOz}$.  Given $\phi$ as a product
Dehn twists around homologically non-trivial curves $a_1, \ldots, a_k$ on a planar
surface $S$ with $n$ boundary components, one first constructs the Stein manifold $S
\times D^2$ in a standard way by attaching $n$ one-handles to $D^4$ along Legendrian
unknots, then one attaches $2$-handles along Legendrian realizations of $a_i$ on $S$
with $\pm 1 $ framing depending on whether the Dehn twist around $a_i$ is negative or
positive. Let $W$ be thus constructed $4$-manifold, then the contact manifold
$(Y,\xi)$ supported by the open book $(S,\phi)$ is the boundary of $W$. As long as
$c_1(\xi)=0$ (or more generally a torsion class) in $H^2(Y)$, $d_3(\xi)$ is an element of
$\f{Q}$ and may be computed by the formula, \[ d_3(\xi) = \frac{1}{4} (c^2(W) -
2\chi(W) -3 \sigma(W)) + q \]

where $q$ is the number of negative Dehn twists in the factorization of $\phi$.
Furthermore, $c^2(W)$ is the square of the class $c(W) \in H^2(W)$ which is Poincar\'e
dual to the class $\Sigma_{i=1}^k \text{rot}(a_i) C_i \in H_2(W,Y)$, where $C_i$ is
the cocore of the 2-handle attached along $a_i$ and $\text{rot}(a_i)$ is the rotation
number of $a_i$ which can be computed as the winding number of $a_i$ with respect to a
standard trivialization of the tangent bundle of the page. Since we assume
$c_1(\xi)=0$ , $c(W)$ maps to zero under the natural map $H^2(W) \to H^2(Y)$, hence it
comes from class in $H^2(W,Y)$ whose square is $c^2(W)$ that appear in the formula
above. 

Now, Figure \ref{Figure7} is a Kirby diagram for the page of a planar open book with
four boundary components. We drew all the curves $a,b,c,d,e$ and $f$ that appear in
the above monodromies. 

\begin{figure}[!h]
\centering
\includegraphics[scale=0.7]{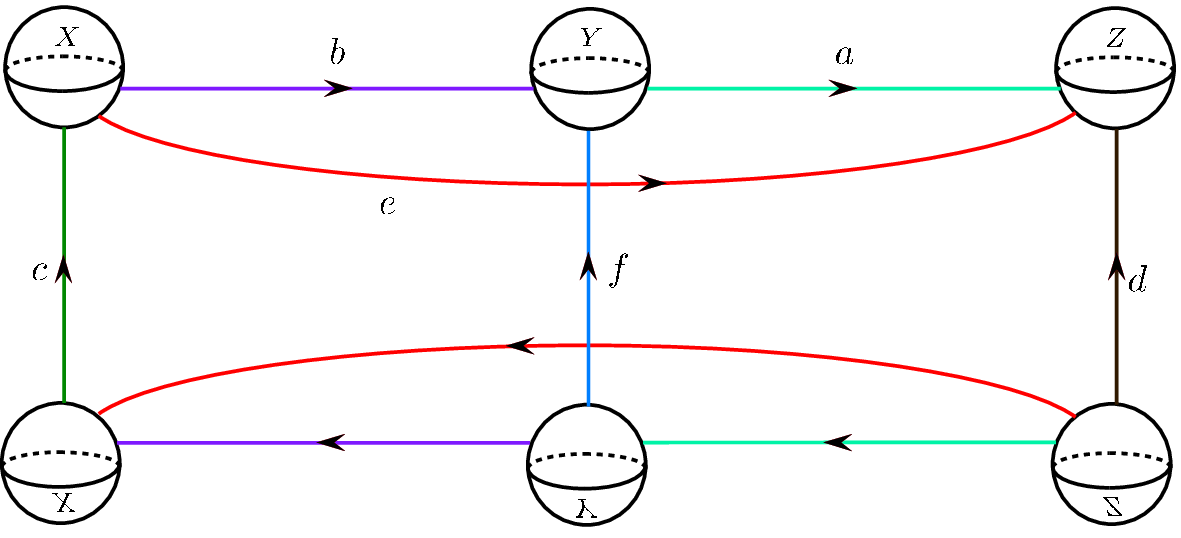} 
\caption{The diagram of the page}
\label{Figure7}
\end{figure}

In order to compute the rotation numbers, we chose an orientation of the curves (note
that the computation of $c^2(W)$ is independent of this choice). One then computes the winding numbers of these curves to get:
\begin{eqnarray*}
\text{rot}(a)=\text{rot}(b)=\text{rot}(e)=1\\ 
\text{rot}(c)=\text{rot}(f)=\text{rot}(d)=0
\end{eqnarray*}
The rest of the proof is a direct homology calculation based on the descriptions
above.

\ \\ 

{\bf Computation of $d_3(\xi_p)$ and $d_3(\bar{\xi}_p)$}

Let $X$, $Y$, and $Z$ be the $1$-handles, which form a basis of $C_1(W; \f{Z})$. Let
$B$, $\{C_1,C_2,C_3\}$, $D$, $\{E_1,E_2, \ldots E_{|p|}\}$, $\{F_1,F_2\}$ be the cores of the
handles attached corresponding to the factorization $\phi_p = \tau_b \tau_c^{-3}
\tau_d \tau_e^{-p} \tau_f^{-2}$. These form a basis of $C_2(W; \f{Z})$ and the
boundary map can be read off the diagram in Figure \ref{Figure7} to be :
\begin{eqnarray*}
d(B) &=& Y-X \ \ d(C_i) = X \ \ \\ 
d(D) &=& Z \ \ d(D-E_i) = X \ \ d(F_i) = Y \ \ 
\end{eqnarray*}
Thus, $H_1(W)=0$ and $H_2(W) = \f{Z}^{|p|+4}$. It will be convenient to pick the following
basis of generators: 
\begin{eqnarray*} \{ C_1 - D + E_{|p|} , C_1 - D + E_{|p|-1}, \ldots, C_1 -D+E_1 , B+C_1 -
F_1 ,B+C_2 -F_1, \\ B+C_3 - F_1 , F_2 - F_1 \}
\end{eqnarray*}

Note that we have $B^2 = D^2=-1$ and $C_1^2=C_2^2=C_3^2=F_1^2= F_2^2= 1$ and $E_i^2 = \sgn(p) $ and any cross term intersection number is zero. The intersection matrix
takes particularly nice form if we add or subtract the first $p$ elements in the above
basis to the $(p+1)^{th}$ element according to whether $p$ is negative or postive. So, our new basis is given by \begin{eqnarray*} \{ C_1 -
D + E_{|p|} , C_1 - D + E_{|p|-1}, \ldots, C_1 -D+E_1 , \\ B+(1-p)C_1+ pD - F_1
- \sgn(p) E_1 -\ldots - \sgn(p) E_{|p|} , \\ B+C_2-F_1, B+C_3 - F_1, F_2 -F_1 \} 
\end{eqnarray*}
Therefore, the intersection matrix of $W$ in the this basis can be calculated to be: 
\[
 Q_W =
 \begin{pmatrix}
  \sgn(p) & 0       & \cdots    & \cdots & \cdots & \cdots & 0  \\
   0        & \ddots  & 0         & \cdots & \cdots & \cdots & 0\\
  \vdots    & 0       & \sgn(p) & 0      & \cdots & \cdots & 0  \\
  \vdots    & \vdots  & 0         & 1-p & 0 &0 & 1 \\
  \vdots    & \vdots  & \vdots    & 0  & 1&0 & 1 \\
   \vdots   & \vdots  & \vdots    & 0  & 0& 1& 1 \\
   0        & 0       & 0         & 1 & 1 &1 &2
 \end{pmatrix}
\]
From this one can easily compute that $\sigma(W)=2+p$, and also we know that
$\chi(W)=|p|+5$. To compute $c^2(W)$, let us denote the cocores by $\check{B}$,
$\{\check{C}_1,\check{C}_2,\check{C}_3\}$, $\check{D}$, $\{\check{E}_1, \ldots
\check{E}_p\}$, $\{\check{F}_1,\check{F}_2\}$. Then from the calculation of rotation
numbers it follows that $c(W)$ is Poincar\'e dual to \[ \check{B}+ \check{E}_1 +
\ldots+ \check{E}_p \] Evaluating $c(W)$ on our basis of $H_2(W)$, we get the vector
$(1,\ldots,1,1-p,1,1,0 )$ hence the Poincar\'e dual to the pull back of $c(W)$ to
$H^2(W,Y)$ is given by $(1,\ldots,1,1-p,1,1,0)^t \cdot (Q_W)^{-1}$, which one can
calculate to be:
\[
 \begin{pmatrix}
  \sgn(p) & 0       & \cdots    & \cdots & \cdots & \cdots & 0  \\
   0        & \ddots  & 0         & \cdots & \cdots & \cdots & 0\\
  \vdots    & 0       & \sgn(p) & 0      & \cdots & \cdots & 0  \\
  \vdots    & \vdots  & 0         & 0 & -1 &-1 & 1 \\
  \vdots    & \vdots  & \vdots    & -1  & p& -1+p & 1-p \\
   \vdots   & \vdots  &  \vdots    & -1  & -1+p & p& 1-p \\
   0        & 0       & 0         & 1 & 1-p &1-p &-1+p
 \end{pmatrix}
\begin{pmatrix}
1 \vspace{.03in} \\ \vdots \\ 1 \vspace{.07in}
\\ 1-p \vspace{.07in} \\ 1 \vspace{.05in} \\ 1 \vspace{.05in} \\0
\end{pmatrix}
\]
Hence \begin{eqnarray*} c^2(W)&=&
(\sgn(p),\ldots,\sgn(p), -2, -2+3p, -2+3p, 3-3p) \cdot (1,\ldots,1,1-p,1,1,0) \\ &=&
9p-6 \end{eqnarray*}
The number of negative Dehn twists is given by $q(W)= 5 + \frac{|p|+p }{2} $.   
Finally,  we compute: \[ d_3(\xi_{p} ) = \frac{1}{4} ( 9p-6 - 2 (|p|+5)
-3(2+p)) + 5 + \frac{|p|+p}{2} = 2p - \frac{1}{2} \] 

This only covers half of the overtwisted contact structures on $S^3$, to get the other
half, we consider $\bar{\xi}_p$. Note, that $\bar{\xi}_p$ is obtained by orientation
reversal. Therefore, we do not need to compute all the above invariants from scratch. Namely,
we have:
\begin{eqnarray*}
c^2(-W) &=& - c^2(W) = -9p+6 \\
\chi(-W) &=& \chi(W) = |p|+5 \\
\sigma(-W) &=& -\sigma(W) = -p-2 \\
q(-W) &=& 2 + \frac{|p|-p}{2} 
\end{eqnarray*}

Therefore, we have:
\[ d_3(\bar{\xi}_p) = \frac{1}{4} (-9p+6 -2(|p|+5) -3(-p-2)) + 2 + \frac{|p|-p}{2} = -2p
+ \frac{5}{2} \]

We determined the binding number of all the overtwisted contact structures. The proof
of Theorem \ref{theorem1} will be completed once we determine the support norm of the
overtwisted contact structures. Note that because all of the overtwisted contact
structures are supported by a planar open book with page a four-holed sphere, we have
that $sn(\xi) \leq 2$ for all $\xi$. We also know that if $d_3(\xi) \neq -\frac{1}{2},
\frac{1}{2}, \frac{3}{2}$, then $bn(\xi)=4$, therefore the only way for these contact
structures to have support norm strictly less than $2$ is when they are supported by
an open book with page a torus with one boundary component. Now, recall the well-known
fact that the only genus one fibred knots on $S^3$ are trefoil and figure-eight knot
and the corresponding open books have monodromy $\tau_a^{\pm 1} \tau_b^{\pm 1}$, where
$a$ and $b$ are standard generators of the homology of the torus, (this follows from
for example \cite{magnusPeluso}). It is now easy to see that $\tau_a\tau_b$,
$\tau_a\tau_b^{-1}$, $\tau_a^{-1}\tau_b$ are obtained by positively stabilizing the
open books with annulus page supporting the
unique tight contact structure, and the overtwisted contact structure
$\xi_{\frac{1}{2}}$, and $\tau_a^{-1}\tau_b^{-1}$ is obtained by negatively
stabilizing $\xi_{\frac{1}{2}}$, hence corresponds to $\xi_{\frac{3}{2}}$. This
completes the proof of Theorem \ref{theorem1}. \QED

\ 

\begin{minipage}[b]{0.7\linewidth}
\centering
\begin{tabular}{c|c|c|c|c|c|}
 \cline{2-6} 
$ $&\small$n<-\frac{1}{2}$& \small$n=-\frac{1}{2}$&\small$n=\frac{1}{2}$
&\small$n=\frac{3}{2}$ &\small$n>\frac{3}{2}$ 
\cr\hline
 \multicolumn{1}{|c |}{\small$bn$}  & \small$4$ & \small$3$ &  \small$2$ & \small$3$ & \small $4$  \cr \hline
  \multicolumn{1}{|c|}{\small$sn$} & \small$2$ & \small$1$ & \small$0$ & \small$1$
&\small $2$\cr \hline
  \multicolumn{1}{|c|}{\small$sg$} & \small$0$ & \small$0$ & \small$0$ &\small$0$
& \small $0$ \cr
\hline
\end{tabular}
\end{minipage}
\hspace{0.5cm}
\begin{minipage}[b]{0.1\linewidth}
\centering
    \begin{tabular}{|c|}
\cline{1-1}
       {\small $\xi_{st}$}
\cr\hline
\multicolumn{1}{|c|} {\small$1$ } 
\cr\hline
\multicolumn{1}{|c|} {\small$-1$}
\cr\hline
\multicolumn{1}{|c|} {\small $0$}
\cr \hline
\end{tabular}
\end{minipage}

\section{Positive factorizations}

In this section we give a proof of Theorem \ref{theorem3} in the following two propositions. 
Recall that for the four-holed sphere $S$, $Map(S,\partial S) = \f{Z}^4 \times F_2$.
The first homology is $H_1(Map(S,\partial S)) = \f{Z}^6$ where the class of a general
element $\phi= \tau_a^\alpha \tau_b^\beta \tau_c^\gamma \tau_d^\delta
\tau_e^{\epsilon_1} \tau_f^{\eta_1} \ldots \tau_e^{\epsilon_k} \tau_f^{\eta_k} $ is
given by $(\alpha,\beta,\gamma,\delta, \sum_{i=1}^k \epsilon_k , \sum_{i=1}^k \eta_k )
$. We will prove the following proposition:

\begin{proposition}

Let $\phi = \tau_a^\alpha \tau_b^\beta \tau_c^\gamma \tau_d^\delta \tau_e^\epsilon
\tau_f^\eta$, then $\phi$ admits a positive factorization if and only if $\min\{\alpha,\beta,\gamma,\delta\} \geq \max \{-\epsilon, -\eta, 0\}$

\end{proposition}

\begin{proof} Suppose $\phi = \tau_{C_1} \ldots \tau_{C_k}$ is a positive
factorization in $Map(S, \partial S)$. We consider the quotient relation in
$H_1(Map(S,\partial S))=\f{Z}^{6}$. 
\[ (\alpha, \beta,\gamma, \delta , \epsilon, \eta) = [\tau_{C_1}] \ldots [\tau_{C_k}]
\] 
Now, by topological classification of surfaces observe that any simple closed curve
$C_i$ is conjugate in $Map(S,\partial S)$ to one of the curves $a,b,c,d,e,f$ or $g$ in
Figure \ref{Figure2}. The classes of Dehn twists around these curves in
$H_1(Map(S,\partial S))$ are given by $[\tau_a]= (1,0,0,0,0,0), [\tau_b]= (0,1,0,0,0,0),
[\tau_c]=(0,0,1,0,0,0), [\tau_d]=(0,0,0,1,0,0), \tau_e=(0,0,0,0,1,0), [\tau_f]=
(0,0,0,0,0,1)$ and $[\tau_g]=(1,1,1,1,-1,-1)$. For the latter, observe that by the
lantern relation we have $\tau_g= \tau_a \tau_b \tau_c \tau_d \tau_e^{-1}
\tau_f^{-1}$. Let us denote by $e_i \in \f{Z}^6$ to be the vector with $i^{th}$
coordinate $1$ and other coordinates $0$ and let $n=(1,1,1,1,-1,-1)$. Therefore, each
class $[\tau_{C_i}]$ is equal to either some $e_j$ or $n$. Now, if $\phi$ has positive
factorization then 
\[ (\alpha, \beta,\gamma, \delta , \epsilon, \eta) = p_0 n + \sum_{i=1}^6 p_i e_i \]

for some $p_i \geq 0$. Thus if $\epsilon$ or $\eta$ is negative, $p_0 \geq 
\max\{-\epsilon, -\eta\}$, which shows that $\min\{\alpha, \beta, \gamma,\delta\} \geq
\max \{-\epsilon, -\eta, 0\}$ as desired.

Conversely, if $\epsilon , \eta \geq 0$, then the given factorization is positive as
long as $ \min\{\alpha,\beta,\gamma,\delta\} >0$. Without loss of generality, suppose
next  that $\epsilon < 0$  and $\eta= \epsilon + r$ for $r\geq 0$. We have
$\min\{\alpha,\beta,\gamma,\delta\} \geq -\epsilon$, set $-\epsilon = k>0$. Then by
using the lantern relation $k$ times, we obtain the central element
$(\tau_f\tau_e\tau_g)^k$. We first use this to kill the negative powers of $f$, to
get:\begin{eqnarray*}
 \phi &=& \tau_a^{\alpha-k} \tau_b^{\beta-k} \tau_c^{\gamma -k } \tau_d^{\delta-k}
\tau_e^{-k} \tau_f^{-k} (\tau_f \tau_e \tau_g)^k \tau_f^r \\
 &=& \tau_a^{\alpha-k} \tau_b^{\beta-k} \tau_c^{\gamma -k } \tau_d^{\delta-k}
\tau_e^{-k} \tau_f^{-k} \tau_f^k (\tau_e \tau_g)^k \tau_f^r \\ 
 &=& \tau_a^{\alpha-k} \tau_b^{\beta-k} \tau_c^{\gamma -k } \tau_d^{\delta-k}
\tau_e^{-k+1} \tau_g (\tau_e \tau_g)^{k-1} \tau_f^r 
\end{eqnarray*}
The proof will be completed once we show that $\tau_e^{-k+1} \tau_g (\tau_e
\tau_g)^{k-1}$ has a positive factorization. We do this by induction and using the
well-known fact that if $f : S\to S$ a diffeomorphism and $C$ a simple closed curve
then the equality $\tau_{f(C)} = f^{-1} \tau_C f$ holds. For $k=1$, the expression is
equal to $\tau_g$ so it is positive. We write \begin{eqnarray*} \tau_e^{-k+1} \tau_g
(\tau_e \tau_g)^{k-1} &=& \tau_e^{-k+1} \tau_g \tau_e^{-(-k+1)} \tau_e^{-k+2} \tau_g
(\tau_e \tau_g)^{k-2} \\ &=& \tau_{\tau_e^{k-1}(g)} \tau_e^{-k+2} \tau_g (\tau_e
\tau_g)^{k-2}   \end{eqnarray*}The latter expression is positive by induction
hypothesis, which completes the proof. In fact, we can simply see that 
\[ \phi =  \tau_a^{\alpha-k} \tau_b^{\beta-k} \tau_c^{\gamma -k } \tau_d^{\delta-k}
\tau_{\tau_e^{k-1}(g)} \tau_{\tau_e^{k-2}(g)} \ldots \tau_{\tau_e(g)}
\tau_g \tau_{f}^r  \]
\end{proof}
\begin{remark}
Suppose more generally that $\phi= \tau_a^\alpha \tau_b^\beta \tau_c^\gamma \tau_d^\delta
\tau_e^{\epsilon_1} \tau_f^{\eta_1} \ldots \tau_e^{\epsilon_k} \tau_f^{\eta_k} $. Let
$\epsilon = \sum_{i=1}^k \epsilon_k $ and $\eta = \sum_{i=1}^k \eta_k$. Then 
the same argument using $H_1(Map(S,\partial S))$ gives that $\phi$ has a positive
factorization only if $\min \{ \alpha , \beta, \gamma, \delta \} \geq \max
\{-\epsilon, -\eta , 0 \} $. However, it is easy to see that this is not sufficient.
For example, $\tau_e \tau_f \tau_e^{-1} \tau_f^{-1}$ satisfies this condition,
but one can check that this is not a right-veering monodoromy hence cannot be
written as a product of right-handed Dehn twists (by \cite{HKM1} the supported contact
structure can not even be tight). On the other hand, the
argument given in the above proof clearly gives a positive factorization of $\phi$ if
$\min\{ \alpha,\beta,\gamma,\delta\} \geq \sum_{i=1}^k \max \{ \epsilon_i , \eta_i,
0\}$.

\end{remark}

We next determine whether or not the Ozsv\'ath-Szab\'o contact invariant vanishes for
these contact structures. In our case, it turns out that this is equivalent to whether
or not the contact structure is Stein fillable. 

\begin{proposition}
Let $\phi = \tau_a^\alpha \tau_b^\beta \tau_c^\gamma \tau_d^\delta \tau_e^\epsilon
\tau_f^\eta$, then the contact invariant $c^+(\phi)$ is non-zero if and only if $\min\{\alpha,\beta,\gamma,\delta\} \geq \max \{-\epsilon, -\eta, 0\}$
\end{proposition}

\begin{proof}

Let us define $\phi_a$, $\phi_b$, $\phi_c$, $\phi_d$ the induced monodromies on the
three-holed planar surface after one ``caps off'' the boundary components parallel to
$a$, $b$ , $c$ or $d$ by gluing a disk to the corresponding boundary component and
extending the monodromy by identity on this disk. Let $\xi_a$, $\xi_b$, $\xi_c$ and
$\xi_d$ be the corresponding contact structures obtained this way.

In \cite{baldwinCap} Corollary 1.3, Baldwin proves that if the contact invariant of
any of the contact structures $\xi_a$, $\xi_b$, $\xi_c$ and $\xi_d$ is zero, then it
must be the case that $c^+(\phi)=0$. Without loss of generality, suppose that
$\epsilon = -k <0$, and $\alpha < k$. Then let's consider the open book $\phi_b$ where
the boundary component parallel to the curve $b$ is capped off. In that case, $e$
becomes isotopic to $a$, and no other curves become isotopic to these pair. Therefore,
the monodromy $\phi_b$ has $k-\alpha$ left-handed Dehn twists around the boundary
component corresponding to $a$, which shows that the supported contact structure
$\xi_b$ is overtwisted \cite{HKM1}. Therefore, the contact invariant $c^+(\xi_b)=0$ by
\cite{OS}.  Hence, it follows that $c^+(\phi)=0$

Conversely, if $\min\{\alpha,\beta,\gamma,\delta\} \geq \max \{-\epsilon, -\eta, 0\}$,
then by the previous proposition, the supported contact structure is Stein fillable,
hence $c^+(\phi) \neq 0$ by \cite{OS}. \end{proof}

\begin{remark} As in the previous remark, for the more general class of
diffeomorphisms, $\phi= \tau_a^\alpha \tau_b^\beta \tau_c^\gamma \tau_d^\delta
\tau_e^{\epsilon_1} \tau_f^{\eta_1} \ldots \tau_e^{\epsilon_k} \tau_f^{\eta_k}$, the
above argument shows that $c^+(\phi)=0$ when $\min \{ \alpha , \beta, \gamma, \delta
\} < \max \{-\epsilon, -\eta , 0 \} $.  However, the converse is not true again by the
same example given there, namely $\tau_e \tau_f \tau_e^{-1} \tau_f^{-1}$ has vanishing
$c^+(\phi)=0$, since it supports an overtwisted contact structure (\cite{OS}). We
would like to point out that this is an overtwisted contact structure on $T^3$.

\end{remark}

\section{An example}
 
\subsection{Poincar\'e homology sphere} 

Once we have the surgery description on pure three-braid closures, it is easy to play
around with simple three-braids and surgeries on them to get interesting open book
decompositions on various manifolds. 

\emph{Proof of Theorem \ref{theorem2}: }  The simplest non-trivial pure three-braid is arguably $\beta = (\sigma_2
\sigma_1 \sigma_2)^2 \sigma_1^{-4}$  whose braid closure is shown in Figure \ref{Figure8}. 

\begin{figure}[!h] \centering
\includegraphics[scale=0.6]{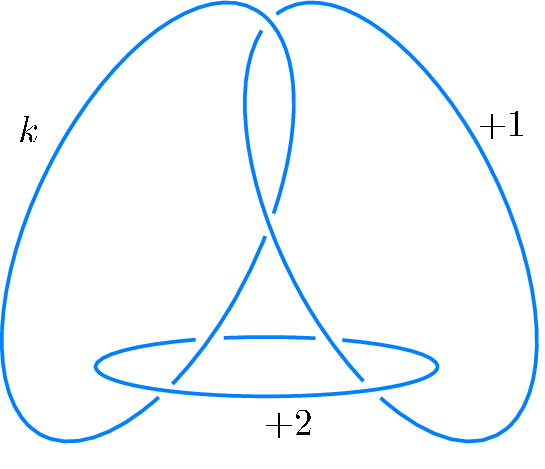} 
\caption{The braid closure of $\beta = (\sigma_2 \sigma_1 \sigma_2)^2 \sigma_1^{-4} $ }
\label{Figure8} \end{figure}

We first verify below that if we do surgery on the closure of this braid with surgery
coefficients $(k,1,2)$, then the result is $k-5$ surgery on the left-handed trefoil.
By Proposition \ref{surgery}, we obtain an open book with page a four-holed sphere and
$\phi = \tau_a^{k+1} \tau_b^2 \tau_c \tau_d \tau_e^{-2}$. In particular, the $(4,1,2)$
surgery yields the infamous Poincar\'e homology sphere which has a unique tight
contact structure.  It is easy to see that this monodromy is right-veering (the
monodromy is given by a multi-curve, and it is easy to see that all the boundary
components are protected. To be more careful, one can apply Corollary 3.4 of
\cite{HKM1} to each boundary component for the monodromy $\tau_a \tau_b \tau_c \tau_d
\tau_e^{-2}$ and use the
fact that the right-veering diffeomorphisms is a submonoid of the mapping class group).

\begin{figure}[!h]
\centering
\includegraphics[scale=0.6]{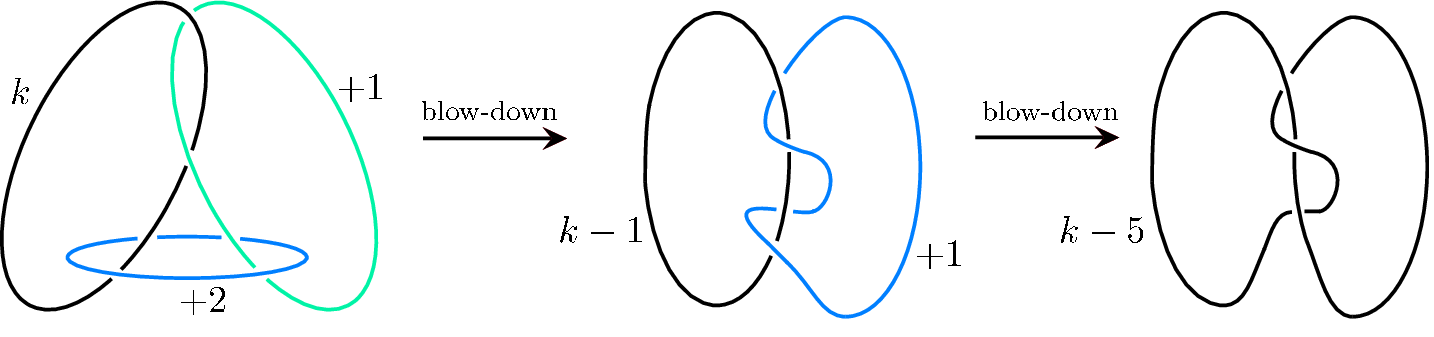} 
\caption{Surgery on the left-handed trefoil}
\label{Figure9}
\end{figure}

We next prove that $(S,\phi)$ is also not destabilizable. If it were, then it would be
a stabilization of an open book $(P,\phi')$, with page $P$ a three-holed sphere, where
$\phi' = \tau_a^p \tau_b^q \tau_c^r$ for the curves $a$, $b$ and $c$ are as in Figure
\ref{Figure1}. As we noted before, such a $(P,\phi')$ is an open book on a Seifert
fibred space with $e_0 = \lfloor -\frac{1}{p} \rfloor + \lfloor -\frac{1}{q} \rfloor +
\lfloor -\frac{1}{r} \rfloor$. Now, since $e_0$ for the Poincar\'e homology sphere is
$-2$, it follows that at least one of the exponents $p$,$q$ or $r$ is negative (for
example, $\tau_a^{-2} \tau_b^3 \tau_c^5$ is an open book on the Poincar\'e homology
sphere). Any stabilization which gives a page with
four holed sphere must be obtained by attaching a Hopf band $h$ to a fixed boundary
component of $P$ and introducing a new monodromy curve $s$ which intersects the cocore
of $h$ at a unique point, so that $\phi = \tau_s \phi' $ , where $\phi'$ is extended
by identity along $h$ to a diffeomorphism of $S$. We now argue that the monodromy of
every such open book is not right-veering. Indeed, by noting that the curve $s$ is
constraint to intersect the cocore of $h$ (which is a properly embedded arc that
connects the top two boundary components of $S$, in Figure \ref{Figure10}), it is easy
to see that one could always find a diffeomorphism of $S$ (not necessarily fixing
boundary components but sending boundary components to boundary components), such that
the configuration of monodoromy curves in $\tau_s \phi'$ is as in Figure
\ref{Figure10}, where the sets of curves $x$, $y$ and $z$ depicted are a permutation
of the images of the sets curves corresponding to Dehn twists around $a$, $b$ and $c$ after stabilization.  Indeed, $S$ is
composed of two pair-of-pants, separated by a curve parallel to $x$ curves in Figure
\ref{Figure10}, and noting the fact that $s$ intersects cocore of $h$ at a unique point, we can
arrange by isotopy so that it intersects the common boundary of the two pair-of-pants
at precisely two points, and now apply a diffeomorphism of the pair-of-pants at the
bottom, which fixes the boundary component parallel to $x$ curves, but rotates the
other boundary components if necessary. Finally, since we know that at least one of
the sets of curves $x$ or $y$ or $z$ are all negative Dehn twists, by looking at
Figure \ref{Figure10}, it is now easy to see that the monodromy $\tau_s \phi'$ is not
right-veering. This proves that $\phi$ cannot be destabilized.

\begin{figure}[!h]
\centering
\includegraphics[scale=1]{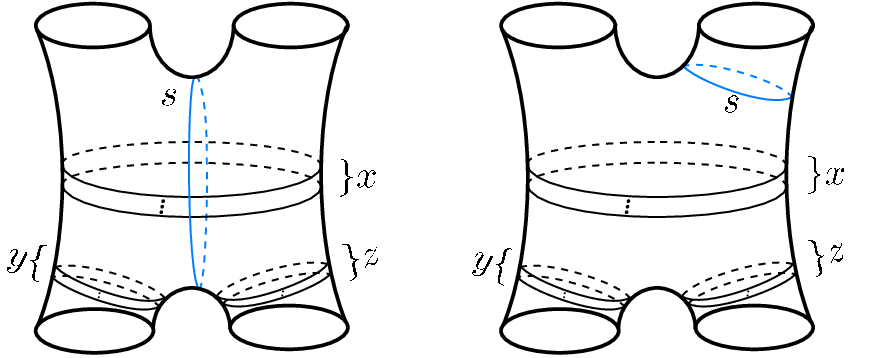} 
\caption{Stabilized open books}
\label{Figure10}
\end{figure}

On the other hand, it follows from the obstruction result of $\cite{etnyre}$ that the
unique tight contact structure on the Poincar\'e homology sphere cannot support a
planar open book.  (Alternatively, it is known that the unique tight contact structure
on the Poincar\'e homology sphere is Stein fillable, hence, by the results of Wendl
\cite{wendl}, if a planar open book supports this contact structure, it should have a
positive factorization but this contradicts our Theorem \ref{theorem3}. )

Therefore, the contact structure supported by the open book $(S, \tau_a^{k+1} \tau_b^2
\tau_c \tau_d \tau_e^{-2})$ is an overtwisted one, which completes the proof of our
Theorem \ref{theorem2}.

\begin{remark}

Note that the way we argued for the overtwistedness of $\tau_a^{5} \tau_b^2 \tau_c
\tau_d \tau_e^{-2}$ is quite special to the case of Poincar\'e homology sphere. In
particular, it used the fact that there exists a unique tight contact structure on the
Poincar\'e homology sphere. Same argument can be made for $k=0,1,2,3$ to obtain
right-veering, not destabilizable monodromies which support overtwisted contact
structures (using the classification result in \cite{ghiggini}, which in particular
says that all the tight contact structures on these manifolds are Stein fillable).
However, we do not know if $\tau_a^{k+1} \tau_b^2 \tau_c \tau_d \tau_e^{-2}$ is
overtwisted or tight for $k >4 $. These are Seifert fibered manifolds $M(-2;
\frac{1}{2}, \frac{2}{3}, \frac{k}{k+1})$, they have $e_0(M)=-2$ but they are not
L-spaces. The classification of tight contact structures on these Seifert manifolds
seems not yet to have been completed.  Note that the corresponding monodromies are all
right-veering diffeomorphisms and the contact invariants of the corresponding contact
structures are zero. 

\end{remark}


\begin{thebibliography}{99}

\bibitem{adamsSherman} C. Adams, W. Sherman, {\sl Minimum ideal triangulations of
hyperbolic 3-manifolds} Discrete Comput. Geom. {\bf 6} (1991) 135--153
\bibitem{baldwinCap} J. Baldwin, {\sl Capping off open books and the Ozsvath-Szabo
contact invariant } arXiv:0901.3797 (2009)
\bibitem{balEt} J. Baldwin, J. Etnyre,
{\sl A note on the support norm of a contact structure} arXiv:0910.5021 (2009)
\bibitem{EtL} T. Etg\"u, Y. Lekili, {\sl Examples of planar tight contact structures
with support norm one} Int. Math. Res. Not. (2010) doi: 10.1093/imrn/rnq025 
\bibitem{etnyrelectures} J. Etnyre, {\sl Lectures on open book decompositions and
contact structures} Clay. Math. Proc. {\bf 5} 103--141
\bibitem{etnyre} J. Etnyre, {\sl Planar open book decompositions and contact
structures} Int. Math. Res. Not. {\bf 2004} (2004) 4255--4267
\bibitem{EtOz} J. Etnyre, B. Ozbagci,
{\sl Invariants of contact structures from open books}, Trans. Amer. Math. Soc.,
360(6):3133--3151, 2008.
\bibitem{eudave}  L. Armas-Sanabria, M. Eudave-Mu$\tilde{\text{n}}$oz, {\sl The hexatangle} Topology
Appl. {\bf 156} (2009) 1037--1053
\bibitem{FadellNeuwirth} E. Fadell, L. Neuwirth, {\sl Configuration Spaces} Math. Scand.
{\bf 10} (1962) 111--118
\bibitem{Fox} R. H. Fox, L. Neuwirth {\sl The braid groups} Math. Scand. {\bf 10} (1962) 119--126
\bibitem{ghiggini} P. Ghiggini {\sl On tight contact structures with negative maximal
twisting number on small Seifert manifolds} Algebr. Geom. Topol. {\bf 8} (2008)
381--396
\bibitem{G} E. Giroux, {\sl G\'{e}ometrie de contact: de la dimension trois
vers les dimensions sup\'{e}rieures,} Proceedings of the
International Congress of Mathematicians {\bf 2} (2002) 405--414
\bibitem{gordon} C. McA. Gordon, {\sl Small surface and Dehn filling} Geometry and
Topology Monographs {\bf 2} (1999) 177--199
\bibitem{gordonwu} C. McA. Gordon, Y. Q. Wu, {\sl Toroidal and annular Dehn fillings}
Proc. London Math. Soc.{\bf 78} (1999) 662--700
\bibitem{HKM1}
K. Honda, W. Kazez, G. Mati\'c, {\sl Right-veering diffeomorphisms of compact surfaces
with boundary}, Inv. Math., {\bf 169} (2007) 427--449
\bibitem{HKM3}
K. Honda, W. Kazez, G. Mati\'c, {\sl On the contact class in Heegard Floer homology}
J. Differential Geom. {\bf 83} (2009) 289--311
\bibitem{magnusPeluso} W. Magnus, A. Peluso, {\sl On Knot Groups} Comm. Pure and Applied
Math. {\bf 20} (1967) 749--770
\bibitem{martelliPetronio} B. Martelli, C Petronio, {\sl Dehn filling of the ``magic''
3-manifold} Comm. Anal. Geom {\bf 14} (2006) 969--1026
\bibitem{murasugi} K. Murasugi, {\sl On closed 3-braids} Number 151 in Memoirs of the American Mathematical Society. American Mathematical Society, 1974.
\bibitem{nieder} K. Niederkr\"uger, C. Wendl, {\sl Weak symplectic fillings and holomorphic curves} arXiv:1003.1923 (2010)
\bibitem{OS} P. Ozsv\'ath, Z. Szab\'o,
{\sl Heegaard Floer homology and contact structures}, Duke Math. J. {\bf 129} (2005)
39--61.
\bibitem{OSS} P. Ozsv\'ath, A. Stipsicz, Z. Szab\'o, {\sl Planar open books and Floer homology}, Int. Math. Res. Not. {\bf 2005} (2005) 3385--3401
\bibitem{prasolov} Prasolov, V. V., Sossinsky, A. B., {\sl Knots, links, braids and
3-manifolds} Translations of Mathematical Monographs, 154. American Mathematical Society, Providence, RI, 1997.
\bibitem{thurston} W. Thurston, {\sl The Geometry and Topology of 3-manifolds} Princeton
University, 1978.
\bibitem{wand} A. Wand, {\sl Mapping class group relations, Stein fillings, and planar open book decompositions} arXiv:10062550 (2010)
\bibitem{wendl} C. Wendl, {\sl Strongly fillable contact manifolds and J-holomorphic
foliations} Duke Math J. {\bf 151} (2010) 337--384
\end{thebibliography}
\end{document}